\documentclass[11pt,a4paper]{amsart}

\usepackage[a4paper,margin=1in]{geometry}

\usepackage{amsmath}
\usepackage{amssymb}
\usepackage{mathtools}
\usepackage{stix2}
\usepackage{microtype}

\usepackage{graphicx}
\usepackage{booktabs}
\usepackage[numbers,sort&compress]{natbib}
\usepackage[colorlinks=true,linkcolor=black,citecolor=black,urlcolor=black]{hyperref}

\numberwithin{equation}{section}
\theoremstyle{plain}
\newtheorem{theorem}{Theorem}[section]
\newtheorem{lemma}[theorem]{Lemma}
\newtheorem{proposition}[theorem]{Proposition}
\newtheorem{corollary}[theorem]{Corollary}
\theoremstyle{definition}
\newtheorem{definition}[theorem]{Definition}
\newtheorem{example}[theorem]{Example}
\theoremstyle{remark}
\newtheorem{remark}[theorem]{Remark}

\newcommand{\Z}{\mathbb Z}
\newcommand{\F}{\mathbb F}
\newcommand{\tr}{\operatorname{tr}}

\newcommand{\lcm}{\operatorname{lcm}}

\newcommand{\cT}{\mathcal T}

\title[Arithmetic landscape functions]{Arithmetic Landscape Functions for the Discrete Cat Map}
\author{Aryaman Chandra}
\address{Heritage Xperiential Learning School, Sector 62, Gurugram, Haryana, India}
\email{aryaman.chandra@gmail.com}
\date{}

\begin{document}
	\begin{abstract}
For a permutation of a finite set, we define a landscape by summing a discounted weight over the times each point returns to itself. The landscape value is determined by the point's least period. We relate the trace of the associated resolvent to fixed-point and exact-period counts, and show that the associated determinant is the reciprocal of the Artin--Mazur dynamical zeta function. We apply this framework to the discrete cat map modulo $N$. For every prime $p\ne5$, all nonzero points have period equal to the order of the cat-map matrix modulo $p$, yielding explicit formulas for the landscape trace and determinant. We prove a period-lifting formula at odd primes $p\ne5$, derive exact fixed-point and period counts modulo $5^a$, and combine prime-power results for composite moduli using the Chinese remainder theorem.
\end{abstract}
	\maketitle
	
	\section{Introduction}
	\label{sec:intro}
	
	Let $T$ be a permutation of a finite set. Every point eventually returns to itself, and the time of its first return is its \emph{period}. The \emph{landscape value} of a point is a discounted count of its returns: a return at time $n$ is worth $\alpha^n$, for a fixed discount factor $0<\alpha<1$. Points on short cycles return often and soon, so they sit high in the landscape; points on long cycles sit low. The landscape is therefore a repackaging of period data, but the repackaging is useful. It packages all orbit counts into one resolvent operator $(I-\alpha P_T)^{-1}$, whose diagonal, trace, determinant and spectrum are all explicit, and which is tied to the dynamical zeta function of $T$.
	
	Our main objects are the permutations $x\mapsto Ax \pmod N$ induced by an invertible integer matrix $A$, and in particular the discrete cat map
	\[
	A=\begin{pmatrix}1&1\\1&2\end{pmatrix}.
	\]
	The paper has two parts.
	
	\begin{itemize}
		\item \emph{General theory (Section~\ref{sec:definition}).} For any finite permutation we prove that the landscape equals $(1-\alpha^k)^{-1}$ at a point of period $k$, that the trace of the resolvent is $\sum_k E(k)/(1-\alpha^k)$, that the determinant of $I-zP_T$ is the reciprocal of the zeta function, and that the spectrum lies on a circle.
		\item \emph{The cat map (Sections~\ref{sec:catmap}--\ref{sec:crt}).} The origin is the unique fixed point and hence the unique maximiser of the landscape. For every prime $p\neq5$ all nonzero points have the same period, so the trace is explicit (Section~\ref{sec:catmap}). Section~\ref{sec:local} proves a period-lifting formula from $p$ to $p^a$ at odd primes not dividing the discriminant. Section~\ref{sec:p5} treats the exceptional prime $5$, which divides the discriminant, and Section~\ref{sec:crt} combines prime powers.
	\end{itemize}
	
	Small cases are worked out as examples (moduli $3$, $5$, $7$, $15$, $49$), and each can be checked by direct enumeration.
	
	\section{The landscape of a permutation}
	\label{sec:definition}
	
	\subsection{Definition and first properties}
	
	Let $T:X\to X$ be a permutation of a finite set $X$ with $|X|=M$. For $x\in X$, write $k_T(x)$ for its least positive period:
	\[
	k_T(x)=\min\{k\geq1:T^k(x)=x\}.
	\]
	Fix $0<\alpha<1$. The \emph{landscape} at $x$ is the sum of $\alpha^n$ over the times $n$ at which the orbit returns to $x$, counting the trivial return at time $0$:
	\[
	\mathcal L_T(x;\alpha)=\sum_{\substack{n\geq0\\T^n(x)=x}}\alpha^n.
	\]
	The return times are $0,k_T(x),2k_T(x),\ldots$, so the sum is geometric.
	
	\begin{theorem}[Period formula]\label{thm:return}
		For every $x\in X$,
		\begin{equation}\label{eq:landscape-period}
			\mathcal L_T(x;\alpha)=\frac{1}{1-\alpha^{k_T(x)}}.
		\end{equation}
	\end{theorem}
	\begin{proof}
		The point returns to itself at times $0,k_T(x),2k_T(x),\ldots$. Therefore
		\[
		\mathcal L_T(x;\alpha)=\sum_{m=0}^{\infty}\alpha^{mk_T(x)}=\frac{1}{1-\alpha^{k_T(x)}},
		\]
		by the geometric-series formula.
	\end{proof}
	
	\begin{example}\label{ex:toy}
		Let $X=\{1,\dots,6\}$ and let $T$ have the cycles $(1\,2\,3)(4\,5)(6)$, that is $1\to2\to3\to1$, $4\to5\to4$, and $6\to6$. Then
		\[
		\mathcal L_T=\frac1{1-\alpha^3}\ \text{on }\{1,2,3\},\qquad
		\mathcal L_T=\frac1{1-\alpha^2}\ \text{on }\{4,5\},\qquad
		\mathcal L_T(6)=\frac1{1-\alpha}.
		\]
		At $\alpha=\tfrac12$ the values are $8/7$, $4/3$ and $2$. The fixed point $6$ is highest, and the longer the cycle, the lower the value. We will reuse this permutation in Sections~\ref{sec:zeta-general} and~\ref{sec:spectrum}.
	\end{example}
	
	The landscape is also a diagonal entry of a resolvent. Define the permutation operator $P_T$ on functions $f:X\to\mathbb C$ by
	\[
	(P_Tf)(x)=f(T^{-1}x),
	\]
	and, for $0<\alpha<1$,
	\[
	G_T(\alpha)=(I-\alpha P_T)^{-1}=\sum_{n=0}^{\infty}\alpha^nP_T^n.
	\]
	The series converges because $P_T$ is a permutation matrix and therefore has operator norm $1$.
	
	\begin{proposition}[Diagonal formula]\label{prop:diagonal}
		For every point $x$,
		\[
		\mathcal L_T(x;\alpha)=[G_T(\alpha)]_{x,x}.
		\]
	\end{proposition}
	\begin{proof}
		The diagonal entry $[P_T^n]_{x,x}$ equals $1$ when $T^n(x)=x$ and $0$ otherwise. Taking the $x$th diagonal entry of the series for $G_T(\alpha)$ gives the return sum that defines $\mathcal L_T(x;\alpha)$.
	\end{proof}
	
	Since $P_T\mathbf 1=\mathbf 1$, we also have $G_T(\alpha)\mathbf 1=(1-\alpha)^{-1}\mathbf 1$.
	
	\begin{proposition}[The landscape determines the period]\label{prop:inverse}
		For fixed $0<\alpha<1$,
		\[
		k_T(x)=\frac{\log\!\left(1-\mathcal L_T(x;\alpha)^{-1}\right)}{\log\alpha}.
		\]
		Consequently, two points have the same landscape value if and only if they have the same period, and every level set $\{x:\mathcal L_T(x;\alpha)=(1-\alpha^k)^{-1}\}=\{x:k_T(x)=k\}$ is a union of complete cycles of length $k$.
	\end{proposition}
	\begin{proof}
		Solve \eqref{eq:landscape-period} for $\alpha^{k}$ to get $\alpha^{k_T(x)}=1-\mathcal L_T(x;\alpha)^{-1}$, and take logarithms. Since $k\mapsto\alpha^k$ is strictly decreasing, different periods give different landscape values. All points of a cycle share the same period, so level sets consist of whole cycles.
	\end{proof}
	
	\begin{corollary}[Extrema]\label{cor:extrema}
		Let $k_{\max}=\max_{x\in X}k_T(x)$. Then
		\[
		\max_{x\in X}\mathcal L_T(x;\alpha)=\frac1{1-\alpha},\qquad
		\min_{x\in X}\mathcal L_T(x;\alpha)=\frac1{1-\alpha^{k_{\max}}}.
		\]
		The maximum is attained exactly at fixed points, and the minimum exactly at points of period $k_{\max}$.
	\end{corollary}
	\begin{proof}
		By Proposition~\ref{prop:inverse}, the landscape value decreases as the period increases. Period one means that the point is fixed, and $X$ is finite, so a largest period exists.
	\end{proof}
	
	Let $E_T(k)$ be the number of points of exact period $k$, let $c_T(k)=E_T(k)/k$ be the number of cycles of length $k$, and let $C_T=\sum_kc_T(k)$ be the total number of cycles. Summing the landscape over $X$ and grouping points by period gives the trace
	\begin{equation}\label{eq:trace-period-sum}
		\tr G_T(\alpha)=\sum_{x\in X}\mathcal L_T(x;\alpha)=\sum_{k\geq1}\frac{E_T(k)}{1-\alpha^k}.
	\end{equation}
	Only finitely many terms are nonzero.
	
	\begin{corollary}[Limit as $\alpha\uparrow1$]\label{cor:alpha-limit}
		For each $x\in X$,
		\[
		\lim_{\alpha\uparrow1}(1-\alpha)\mathcal L_T(x;\alpha)=\frac1{k_T(x)}.
		\]
		Consequently
		\[
		\lim_{\alpha\uparrow1}(1-\alpha)\,\tr G_T(\alpha)=\sum_{x\in X}\frac1{k_T(x)}=C_T,
		\]
		the total number of cycles.
	\end{corollary}
	\begin{proof}
		For every positive integer $k$, $1-\alpha^k=(1-\alpha)(1+\alpha+\cdots+\alpha^{k-1})$, so $(1-\alpha)/(1-\alpha^k)\to1/k$. Summing over the finitely many points gives the second limit. A cycle of length $k$ has $k$ points, each contributing $1/k$, so each cycle contributes exactly $1$.
	\end{proof}
	
	\begin{remark}
		In operator terms, $\tr G_T(\alpha)=\sum_\omega (1-\alpha\omega)^{-1}$ over the eigenvalues $\omega$ of $P_T$, and only $\omega=1$ blows up as $\alpha\uparrow1$. Its multiplicity is the number of cycles, because the indicator function of each cycle is an invariant function. This explains why \emph{every} cycle, and not only the fixed points, contributes to the leading divergence of the trace.
	\end{remark}
	
	\subsection{Fixed-point counts and the zeta function}\label{sec:zeta-general}
	
	The trace is governed by orbit counts, and these are in turn determined by fixed-point counts. Let $F_T(n)=\#\{x\in X:T^n(x)=x\}$ for $n\geq1$, and set $F_T(0)=M$. A point is fixed by $T^n$ exactly when its period divides $n$, so
	\begin{equation}\label{eq:fixed-divisor}
		F_T(n)=\sum_{k\mid n}E_T(k)=\sum_{k\mid n}k\,c_T(k).
	\end{equation}
	
	\begin{theorem}[Exact-period counts]\label{thm:mobius}
		For every $k\geq1$,
		\begin{equation}\label{eq:mobius}
			E_T(k)=\sum_{d\mid k}\mu(k/d)F_T(d),
			\qquad
			c_T(k)=\frac1k\sum_{d\mid k}\mu(k/d)F_T(d),
		\end{equation}
		where $\mu$ is the Möbius function.
	\end{theorem}
	\begin{proof}
		Apply Möbius inversion~\cite{HardyWright1979} to \eqref{eq:fixed-divisor}. Every cycle of length $k$ contains exactly $k$ points of period $k$, so $c_T(k)=E_T(k)/k$.
	\end{proof}
	
	For $|z|<1$, define the trace generating function and the dynamical zeta function
	\[
	\mathcal T_T(z)=\tr(I-zP_T)^{-1},\qquad
	\zeta_T(z)=\exp\!\left(\sum_{n=1}^{\infty}\frac{F_T(n)}{n}z^n\right),
	\]
	together with the determinant $D_T(z)=\det(I-zP_T)$. (The function $\zeta_T$ is the Artin--Mazur zeta function of $T$.) Since $\tr(P_T^n)=F_T(n)$, the Neumann series gives
	\begin{equation}\label{eq:tracegf}
		\mathcal T_T(z)=\sum_{n=0}^{\infty}F_T(n)z^n=M+\sum_{n=1}^{\infty}F_T(n)z^n.
	\end{equation}
	Thus the trace records every fixed-point count as a coefficient. For real $z=\alpha\in(0,1)$ this is the trace of $G_T(\alpha)$ from \eqref{eq:trace-period-sum}.
	
	\begin{theorem}[Zeta product and trace formula]\label{thm:zeta}
		For $|z|<1$,
		\begin{align}
			\zeta_T(z)&=\prod_{k\geq1}(1-z^k)^{-c_T(k)},\label{eq:zeta-product}\\
			D_T(z)&=\prod_{k\geq1}(1-z^k)^{c_T(k)}=\zeta_T(z)^{-1},\label{eq:det-zeta}\\
			\mathcal T_T(z)&=M+z\frac{\zeta_T'(z)}{\zeta_T(z)}
			=M-z\frac{D_T'(z)}{D_T(z)}.\label{eq:trace-zeta}
		\end{align}
		Only finitely many $c_T(k)$ are nonzero.
	\end{theorem}
	\begin{proof}
		On a cycle of length $k$, $\det(I-zP_T)=1-z^k$. Taking the product over cycles gives the product for $D_T$. For the zeta function, expand $-\log(1-z^k)=\sum_{m\geq1}z^{km}/m$. Summing over cycles and collecting the coefficient of $z^n$ gives
		\[
		\sum_{k\geq1}c_T(k)\sum_{m\geq1}\frac{z^{km}}m
		=\sum_{n\geq1}\frac{z^n}{n}\sum_{k\mid n}k\,c_T(k)
		=\sum_{n\geq1}\frac{F_T(n)}n z^n,
		\]
		by \eqref{eq:fixed-divisor}. Exponentiating gives the product for $\zeta_T$ and shows $D_T=\zeta_T^{-1}$. Differentiating the defining series gives $z\zeta_T'/\zeta_T=\sum_{n\geq1}F_T(n)z^n$, and \eqref{eq:tracegf} then yields \eqref{eq:trace-zeta}.
	\end{proof}
	
	\begin{example}\label{ex:toy-zeta}
		For the permutation of Example~\ref{ex:toy}, $c_T(1)=c_T(2)=c_T(3)=1$, so $D_T(z)=(1-z)(1-z^2)(1-z^3)$. Then
		\[
		-z\frac{D_T'(z)}{D_T(z)}=\frac{z}{1-z}+\frac{2z^2}{1-z^2}+\frac{3z^3}{1-z^3},
		\]
		and \eqref{eq:trace-zeta} gives
		\[
		\mathcal T_T(z)=6+\frac{z}{1-z}+\frac{2z^2}{1-z^2}+\frac{3z^3}{1-z^3}
		=\frac1{1-z}+\frac2{1-z^2}+\frac3{1-z^3},
		\]
		since $6=1+2+3$. This agrees with \eqref{eq:trace-period-sum}, because $E_T(1)=1$, $E_T(2)=2$ and $E_T(3)=3$. The first fixed-point counts are $F_T(1)=1$ (only the point $6$), $F_T(2)=3$ (the points $4,5,6$) and $F_T(3)=4$ (the points $1,2,3,6$), matching the coefficients of $z,z^2,z^3$ in the expansion above.
	\end{example}
	
	The determinant and the zeta function contain the same orbit information, with reciprocal conventions. The identity \eqref{eq:trace-zeta} is a finite, equal-weight periodic-orbit trace formula. It is related in spirit to the trace formulas of Gutzwiller and Ruelle, although here it follows directly from finite permutation dynamics and requires no stability weights~\cite{Gutzwiller1971,Ruelle1994}.
	
	\subsection{Spectral structure}\label{sec:spectrum}
	
	The eigenvalues of $I-\alpha P_T$ are determined by the cycle lengths.
	
	\begin{theorem}[Cycle spectrum]\label{thm:spectrum}
		Suppose the cycles of $T$ have lengths $d_1,\ldots,d_s$. A cycle of length $d_j$ contributes the eigenvalues
		\[
		\{1-\alpha\omega:\omega^{d_j}=1\}
		\]
		to $I-\alpha P_T$. Thus every eigenvalue $\lambda$ lies on the circle $|\lambda-1|=\alpha$. For a root of unity $\omega$ of exact order $m$, the eigenvalue $1-\alpha\omega$ occurs once for each cycle whose length is divisible by $m$.
	\end{theorem}
	\begin{proof}
		Each cycle carries an independent invariant subspace of functions. On a cycle of length $d_j$, $P_T$ acts as a cyclic shift, whose eigenvalues are the $d_j$-th roots of unity, each occurring once. Hence $I-\alpha P_T$ has eigenvalues $1-\alpha\omega$ there. A root of exact order $m$ is a $d_j$-th root of unity exactly when $m\mid d_j$.
	\end{proof}
	
	\begin{corollary}[Cycle recovery]\label{cor:det-cycle-recovery}
		Let $f_T(m)$ be the number of cycles whose lengths are divisible by $m$. Then
		\[
		f_T(m)=\sum_{j\geq1}c_T(mj),\qquad
		c_T(k)=\sum_{j\geq1}\mu(j)f_T(kj).
		\]
		Thus $D_T(z)=\det(I-zP_T)$ determines the number of cycles of every length.
	\end{corollary}
	\begin{proof}
		For a cycle of length $d$, the factor $1-z^d$ contains each cyclotomic polynomial $\Phi_m(z)$ with $m\mid d$, up to an overall sign. The exponent of $\Phi_m$ in $D_T$ is therefore $f_T(m)$. Möbius inversion over multiples gives $c_T(k)$. All sums are finite.
	\end{proof}
	
	\begin{example}
		For the permutation of Example~\ref{ex:toy}, the eigenvalue $1-\alpha$ ($\omega=1$) has multiplicity $3$, one for each cycle. The eigenvalue $1+\alpha$ ($\omega=-1$) has multiplicity $1$, coming from the $2$-cycle. The two eigenvalues $1-\alpha e^{\pm2\pi i/3}$ each have multiplicity $1$, coming from the $3$-cycle. The total is $3+1+2=6=|X|$.
	\end{example}
	
	\section{The discrete cat map}
	\label{sec:catmap}
	
	\subsection{The map and its landscape}
	
	We now specialise to integer matrices modulo $N$.
	
	\begin{definition}[Arithmetic landscape function]
		Let $A$ be a $d\times d$ integer matrix with $\det A=\pm1$, and let $N\geq2$. The inverse matrix $A^{-1}$ has integer entries, so the map
		\[
		T_{A,N}(x)=Ax\pmod N
		\]
		is a permutation of $X_N=(\mathbb Z/N\mathbb Z)^d$. We call $\mathcal L_{T_{A,N}}(x;\alpha)$ the \emph{arithmetic landscape function} of $A$ modulo $N$ and write it $\mathcal L_{A,N}(x;\alpha)$.
	\end{definition}
	
	From now on, until Section~\ref{sec:crt}, we work with the discrete cat map
	\[
	A=\begin{pmatrix}1&1\\1&2\end{pmatrix},\qquad T_N(x)=Ax\pmod N,\qquad X_N=(\mathbb Z/N\mathbb Z)^2,
	\]
	which has $M=N^2$ points. We write $P_N=P_{T_N}$, $\mathcal L_N=\mathcal L_{T_N}$, and similarly $E_N$, $c_N$, $C_N$, $F_N$, $\cT_N$, $D_N$ and $\zeta_N$ for the quantities of Section~\ref{sec:definition} attached to $T_N$. Let $\pi(N)$ denote the order of $A$ modulo $N$, that is, the least positive integer with $A^{\pi(N)}\equiv I\pmod N$. Every point period divides $\pi(N)$.
	
	Cat-map periods and their Fibonacci connections have been studied extensively, including their distributions modulo prime powers~\cite{DysonFalk1992,PercivalVivaldi1987,Chen2012,PeriodDistribution2014}. Here we use the period data to describe the landscape.
	
	\begin{theorem}[Fibonacci identity]\label{thm:fib-power}
		Let $F_0=0$, $F_1=1$, and $F_{n+2}=F_{n+1}+F_n$ for $n\geq0$. Set $F_{-1}=1$ so the matrix formula also holds at $n=0$. Then, for every $n\geq0$,
		\[
		A^n=\begin{pmatrix}F_{2n-1}&F_{2n}\\F_{2n}&F_{2n+1}\end{pmatrix}.
		\]
	\end{theorem}
	\begin{proof}
		Set $C=\begin{pmatrix}0&1\\1&1\end{pmatrix}$. We prove by induction on $m\geq0$ that
		\[
		C^m=\begin{pmatrix}F_{m-1}&F_m\\F_m&F_{m+1}\end{pmatrix}.
		\]
		For $m=0$ this is the identity matrix because $F_{-1}=F_1=1$ and $F_0=0$. If the formula holds for $m$, multiplying by $C$ gives
		\[
		C^{m+1}=\begin{pmatrix}F_m&F_{m-1}+F_m\\F_{m+1}&F_m+F_{m+1}\end{pmatrix}
		=\begin{pmatrix}F_m&F_{m+1}\\F_{m+1}&F_{m+2}\end{pmatrix}.
		\]
		Since $A=C^2$, taking $m=2n$ gives the stated formula.
	\end{proof}
	
	\begin{proposition}[Unique fixed point]\label{prop:unique-fixed}
		For every $N\geq2$, the origin is the unique fixed point of $T_N$.
	\end{proposition}
	\begin{proof}
		A fixed point satisfies $(A-I)x\equiv0\pmod N$. Since $A-I=\begin{pmatrix}0&1\\1&1\end{pmatrix}$ has determinant $-1$, it is invertible modulo every $N$, so $x=0$ is the only solution.
	\end{proof}
	
	\begin{corollary}[Landscape maximum]\label{cor:unique-max}
		For all $N\geq2$ and $0<\alpha<1$,
		\[
		\mathcal L_N(0;\alpha)=\frac1{1-\alpha},\qquad
		\mathcal L_N(x;\alpha)\leq\frac1{1-\alpha^2}\quad(x\ne0).
		\]
		The origin is the unique maximiser, and the gap obeys
		\[
		\mathcal L_N(0;\alpha)-\max_{x\ne0}\mathcal L_N(x;\alpha)\geq\frac{\alpha}{1-\alpha^2}.
		\]
	\end{corollary}
	\begin{proof}
		The origin has period one, and every other point has period at least two because there are no other fixed points. Apply Proposition~\ref{prop:inverse}. The gap estimate follows from $\frac1{1-\alpha}-\frac1{1-\alpha^2}=\frac{\alpha}{1-\alpha^2}$.
	\end{proof}
	
	\begin{example}[The cat map modulo $3$]\label{ex:N3}
		Modulo $3$ the map is $(x,y)\mapsto(x+y,\,x+2y)$. Besides the origin, the eight nonzero points form two cycles of length $4$:
		\[
		(1,0)\to(1,1)\to(2,0)\to(2,2)\to(1,0),\qquad
		(0,1)\to(1,2)\to(0,2)\to(2,1)\to(0,1).
		\]
		So $\pi(3)=4$, and the landscape takes just two values:
		\[
		\mathcal L_3(0;\alpha)=\frac1{1-\alpha},\qquad
		\mathcal L_3(x;\alpha)=\frac1{1-\alpha^4}\quad(x\neq0).
		\]
		At $\alpha=\tfrac12$ these are $2$ and $16/15$. We return to this example in Section~\ref{sec:primes-trace}.
	\end{example}
	
	Figures~\ref{fig:heatmaps} and~\ref{fig:enhancement} illustrate how the landscape levels reflect the orbit periods and how the normalized values approach the reciprocal-period profile of Corollary~\ref{cor:alpha-limit}.
	
	\begin{figure}[htbp]
		\centering
		\includegraphics[width=0.95\linewidth]{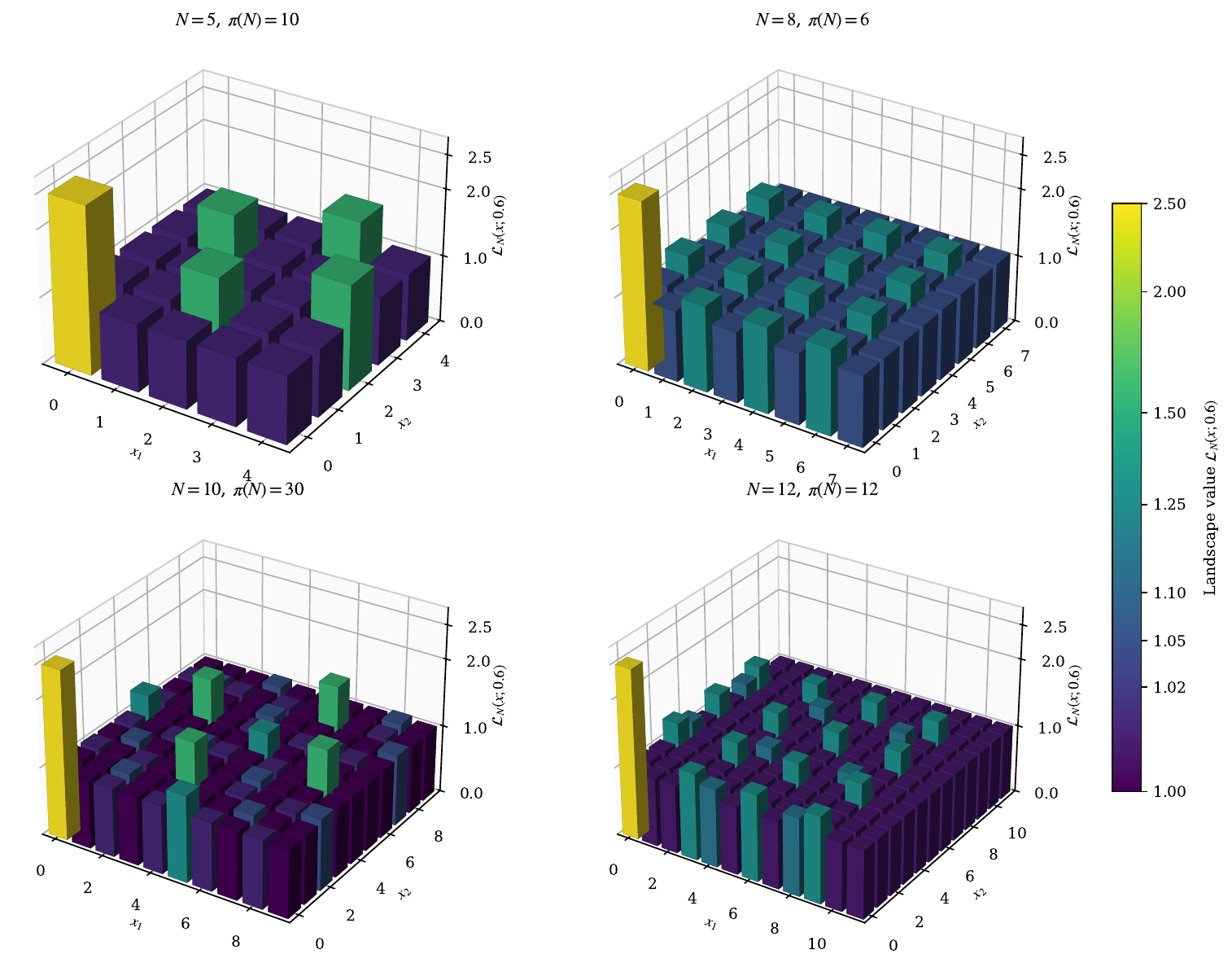}
		\caption{Three-dimensional bar plots of the arithmetic landscape $\mathcal L_N(x;0.6)$ for $N\in\{5,8,10,12\}$. Each bar represents a residue class $x=(x_1,x_2)$, and its height is the landscape value. The origin is the unique maximiser in each panel. A common colour scale is used across the four moduli.}
		\label{fig:heatmaps}
	\end{figure}
	
	\begin{figure}[htbp]
		\centering
		\includegraphics[width=0.95\linewidth]{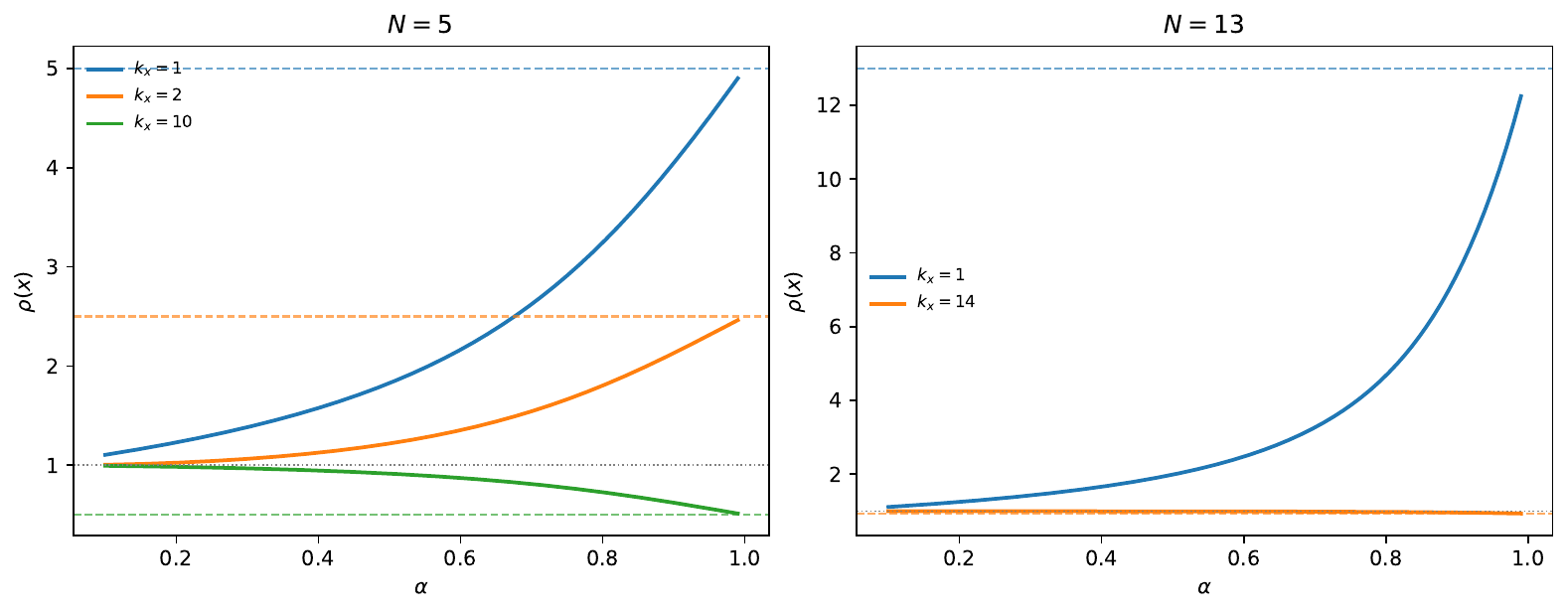}
		\caption{Normalized landscape $\rho(x)=\mathcal L_N(x;\alpha)/\langle\mathcal L_N\rangle_N$ for each period class at $N=5$ and $N=13$. Dashed lines show the limits $\bar{k}/k_x$ as $\alpha\uparrow1$, where $\bar{k}$ is the harmonic mean of the point periods.}
		\label{fig:enhancement}
	\end{figure}
	
	\subsection{Fixed-point counts}
	
	\begin{lemma}[Kernel count]\label{lem:snf-kernel}
		Let $M$ be a $2\times2$ integer matrix with nonzero determinant and Smith invariant factors $d_1\mid d_2$. The number of solutions of $Mx\equiv0\pmod N$ in $(\Z/N\Z)^2$ is
		\[
		\gcd(d_1,N)\gcd(d_2,N).
		\]
	\end{lemma}
	\begin{proof}
		There are integer matrices $U,V$ with determinant $\pm1$ such that $UMV=\operatorname{diag}(d_1,d_2)$. Both are invertible modulo $N$. With $x=Vy$, the equation becomes $d_iy_i\equiv0\pmod N$ for $i=1,2$. Each congruence has exactly $\gcd(d_i,N)$ solutions.
	\end{proof}
	
	\begin{theorem}[Fixed-point count]\label{thm:snf-fixed}
		For $n\geq1$, let $d_1^{(n)}\mid d_2^{(n)}$ be the Smith invariant factors of $A^n-I$. They are positive because $A$ has no eigenvalue that is a root of unity. Then
		\[
		F_N(n)=\gcd(d_1^{(n)},N)\gcd(d_2^{(n)},N).
		\]
	\end{theorem}
	\begin{proof}
		Apply Lemma~\ref{lem:snf-kernel} to $M=A^n-I$.
	\end{proof}
	
	The next lemma is the key to the prime case. It holds for any $A\in SL_2(\Z)$, not only for the cat map. Write $\Delta_A=(\tr A)^2-4$ for the discriminant of the characteristic polynomial $t^2-(\tr A)t+1$.
	
	\begin{lemma}[Dichotomy modulo $p$]\label{lem:diag-mod-p}
		Let $A\in SL_2(\Z)$ and let $p$ be a prime with $p\nmid\Delta_A$. Then $A$ is diagonalizable over $\F_{p^2}$ with eigenvalues $\lambda,\lambda^{-1}$, and for every $n\geq1$ either $A^n\equiv I\pmod p$ or $A^n-I$ is invertible modulo $p$.
	\end{lemma}
	\begin{proof}
		The eigenvalues $\lambda,\lambda^{-1}$ are the roots of $t^2-(\tr A)t+1$ in $\F_{p^2}$, and $(\lambda-\lambda^{-1})^2=(\tr A)^2-4=\Delta_A\not\equiv0$. So they are distinct and $A$ is diagonalizable, with $A^n=\operatorname{diag}(\lambda^n,\lambda^{-n})$ in a suitable basis. If $\lambda^n=1$, then also $\lambda^{-n}=1$ and $A^n=I$. Otherwise both $\lambda^n-1$ and $\lambda^{-n}-1$ are nonzero, so $A^n-I$ is invertible.
	\end{proof}
	
	The cat map has $\Delta_A=5$, so the lemma applies to every prime $p\neq5$, including $p=2$.
	
	\begin{theorem}[Fixed points modulo $p\neq5$]\label{thm:prime-fixed}
		Let $p\neq5$ be prime and $n\geq1$. Then
		\[
		F_p(n)=\begin{cases}
			p^2,&\pi(p)\mid n,\\
			1,&\pi(p)\nmid n.
		\end{cases}
		\]
	\end{theorem}
	\begin{proof}
		If $\pi(p)\mid n$, then $A^n\equiv I$ and all $p^2$ points are fixed. If $\pi(p)\nmid n$, then $A^n\not\equiv I$, so $A^n-I$ is invertible by Lemma~\ref{lem:diag-mod-p}, and only the origin is fixed.
	\end{proof}
	
	\begin{corollary}[One period per prime]\label{cor:one-period}
		For every prime $p\neq5$, every nonzero point modulo $p$ has period exactly $\pi(p)$. Hence $T_p$ has one fixed point and $(p^2-1)/\pi(p)$ cycles of length $\pi(p)$. In particular $\pi(p)\mid p^2-1$, and
		\[
		E_p(1)=1,\qquad E_p(\pi(p))=p^2-1,
		\]
		with all other counts zero, and $C_p=1+(p^2-1)/\pi(p)$.
	\end{corollary}
	\begin{proof}
		If $x\ne0$ satisfies $A^nx=x$, then $A^n-I$ is singular, so $A^n=I$ by Lemma~\ref{lem:diag-mod-p}, and $\pi(p)\mid n$. Since $\pi(p)$ is itself a return time for every point, every nonzero point has period $\pi(p)$. The $p^2-1$ nonzero points split into cycles of equal length $\pi(p)$, so $\pi(p)\mid p^2-1$.
	\end{proof}
	
	\begin{remark}
		The general prime-modulus count has a middle case, $F_p(n)=p$, but it occurs only when $A$ is not diagonalizable modulo $p$. For the cat map this happens only at $p=5$, where $A\equiv-I+(\text{nilpotent})$. We treat $5$ in Section~\ref{sec:p5}; it has fixed-point counts $1$, $5$ and $25$.
	\end{remark}
	
	Table~\ref{tab:cycles} lists representative cycle structures at small moduli, including those used in the trace and spectrum figures, together with $N=2$ and $N=3$. For $N=5$ the data come from Section~\ref{sec:p5}.
	
	\begin{table}[htbp]
		\centering
		\begin{tabular}{@{}rlr@{}}
			\toprule
			$N$ & cycle lengths (number of cycles) & $C_N$\\
			\midrule
			$2$ & $1\ (1),\ 3\ (1)$ & $2$\\
			$3$ & $1\ (1),\ 4\ (2)$ & $3$\\
			$5$ & $1\ (1),\ 2\ (2),\ 10\ (2)$ & $5$\\
			$7$ & $1\ (1),\ 8\ (6)$ & $7$\\
			$11$ & $1\ (1),\ 5\ (24)$ & $25$\\
			$13$ & $1\ (1),\ 14\ (12)$ & $13$\\
			\bottomrule
		\end{tabular}
		\caption{Cycle structure of the cat map modulo small primes. For $p\neq5$ there is one fixed point and $(p^2-1)/\pi(p)$ cycles of length $\pi(p)$.}
		\label{tab:cycles}
	\end{table}
	
	\subsection{Trace, determinant and spectrum at primes}\label{sec:primes-trace}
	
	Corollary~\ref{cor:one-period} makes the general formulas of Section~\ref{sec:definition} explicit at primes.
	
	\begin{theorem}[Trace at primes]\label{thm:large-N}
		Let $p\neq5$ be prime. Then for $|z|<1$,
		\[
		\cT_p(z)=\frac1{1-z}+\frac{p^2-1}{1-z^{\pi(p)}},\qquad
		D_p(z)=(1-z)\bigl(1-z^{\pi(p)}\bigr)^{(p^2-1)/\pi(p)}.
		\]
		Consequently, for fixed $\alpha\in(0,1)$,
		\[
		0\leq\frac{\cT_p(\alpha)}{p^2}-1\leq\frac{\alpha^{\pi(p)}+p^{-2}}{1-\alpha},
		\]
		and in particular $\cT_p(\alpha)/p^2\to1$ as $p\to\infty$.
	\end{theorem}
	\begin{proof}
		The first two formulas follow from Corollary~\ref{cor:one-period} and \eqref{eq:trace-period-sum}, \eqref{eq:det-zeta}. Write $\pi=\pi(p)$. Then
		\[
		\frac{\cT_p(\alpha)}{p^2}-1=\frac{\alpha^{\pi}}{1-\alpha^{\pi}}+\frac1{p^2}\left(\frac1{1-\alpha}-\frac1{1-\alpha^{\pi}}\right).
		\]
		Both terms are nonnegative, and since $1-\alpha^\pi\geq1-\alpha$, they are bounded by $\alpha^\pi/(1-\alpha)$ and $p^{-2}/(1-\alpha)$ respectively. It remains to show that $\pi(p)\to\infty$. Since $A^{\pi}\equiv I\pmod p$, the prime $p$ divides the integer
		\[
		\det(A^\pi-I)=2-\tr(A^\pi)=2-\lambda^\pi-\lambda^{-\pi},\qquad\lambda=\tfrac{3+\sqrt5}{2},
		\]
		which is nonzero, with absolute value less than $\lambda^\pi$. Hence $p<\lambda^{\pi(p)}$, that is, $\pi(p)>\log p/\log\lambda\to\infty$.
	\end{proof}
	
	\begin{example}[Continuation of Example~\ref{ex:N3}]\label{ex:N3-cont}
		For $N=3$ we have $\pi(3)=4$ and two $4$-cycles, so Theorem~\ref{thm:large-N} gives
		\[
		D_3(z)=(1-z)(1-z^4)^2,\qquad
		\cT_3(z)=\frac1{1-z}+\frac8{1-z^4}.
		\]
		As a check of \eqref{eq:trace-zeta}, since $-z\,D_3'/D_3=\frac z{1-z}+\frac{8z^4}{1-z^4}$, we get $9+\frac z{1-z}+\frac{8z^4}{1-z^4}$, which equals the formula above because $9=1+8$. The fixed-point counts are $F_3(n)=9$ if $4\mid n$ and $F_3(n)=1$ otherwise. For the spectrum, the eigenvalue $1-\alpha$ has multiplicity $C_3=3$, the eigenvalue $1+\alpha$ has multiplicity $2$, and each of $1\mp i\alpha$ has multiplicity $2$, for a total of $9$. Finally $(1-\alpha)\cT_3(\alpha)\to3=C_3$ as $\alpha\uparrow1$, in agreement with Corollary~\ref{cor:alpha-limit}.
	\end{example}
	
	For the cat map the unique fixed point does not by itself determine the trace limit: every cycle contributes to the leading divergence, as Corollary~\ref{cor:alpha-limit} and the remark after it explain. Figure~\ref{fig:trace} illustrates the limit for several moduli.
	\begin{figure}[htbp]
		\centering
		\includegraphics[width=0.98\linewidth]{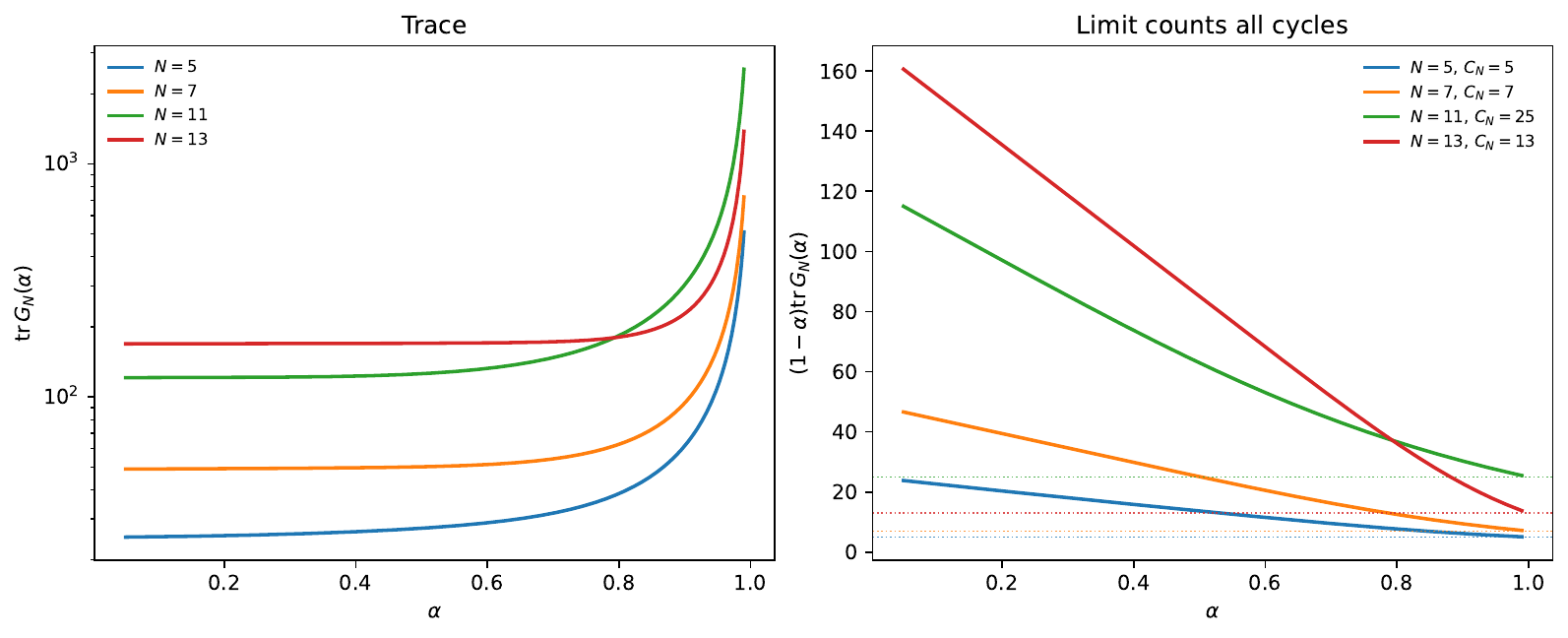}
		\caption{The landscape trace for $N\in\{5,7,11,13\}$. The right panel shows that $(1-\alpha)\tr G_N(\alpha)$ tends to $C_N$, the total number of cycles, rather than to the number of fixed points alone.}
		\label{fig:trace}
	\end{figure}
	
	Figure~\ref{fig:zeta} compares the cycle-product zeta function with the reciprocal determinant and checks the trace--zeta relation \eqref{eq:trace-zeta} using direct matrix calculations for $N=5$ and $N=7$.
	\begin{figure}[htbp]
		\centering
		\includegraphics[width=0.98\linewidth]{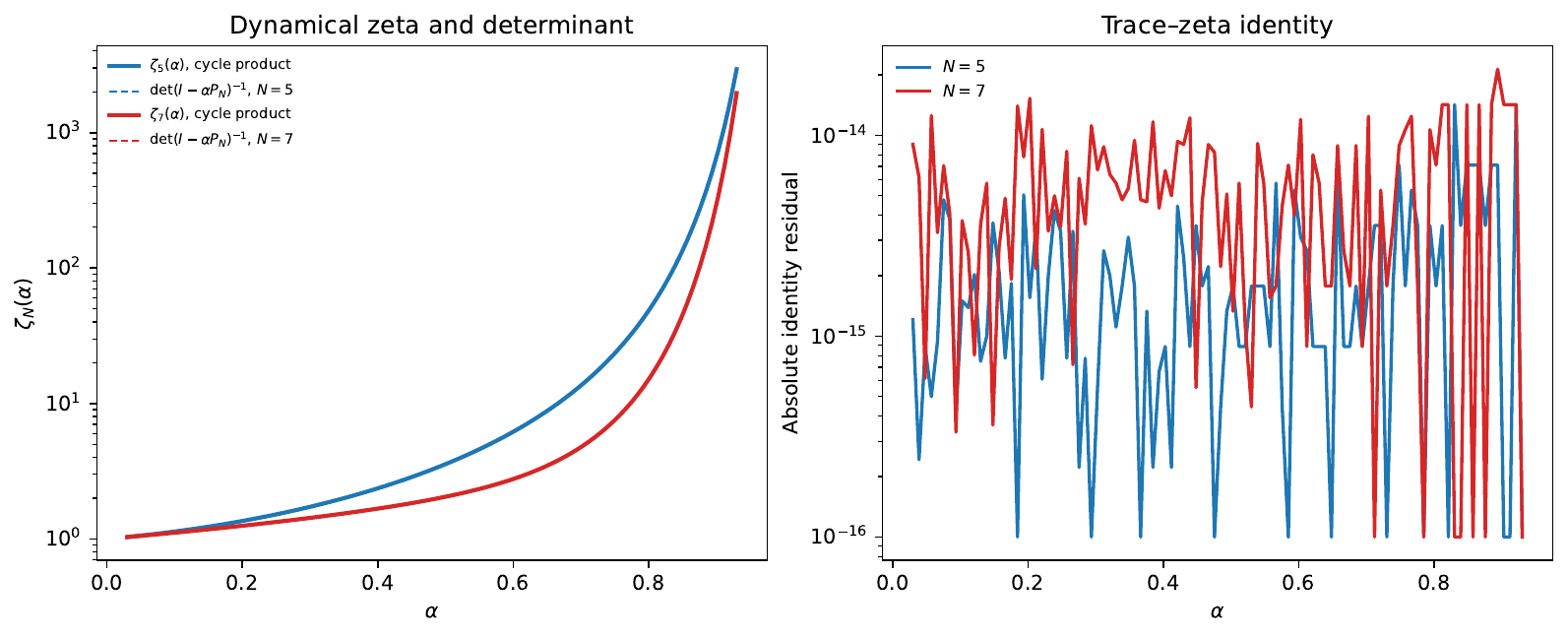}
		\caption{The dynamical zeta function and trace relation for $N\in\{5,7\}$. Left: the cycle product $\zeta_N(\alpha)$ agrees with $\det(I-\alpha P_N)^{-1}$. Right: the numerical residual in \eqref{eq:trace-zeta}, evaluated from the matrix resolvent and the orbit product.}
		\label{fig:zeta}
	\end{figure}
	
	\begin{figure}[htbp]
		\centering
		\includegraphics[width=0.9\linewidth]{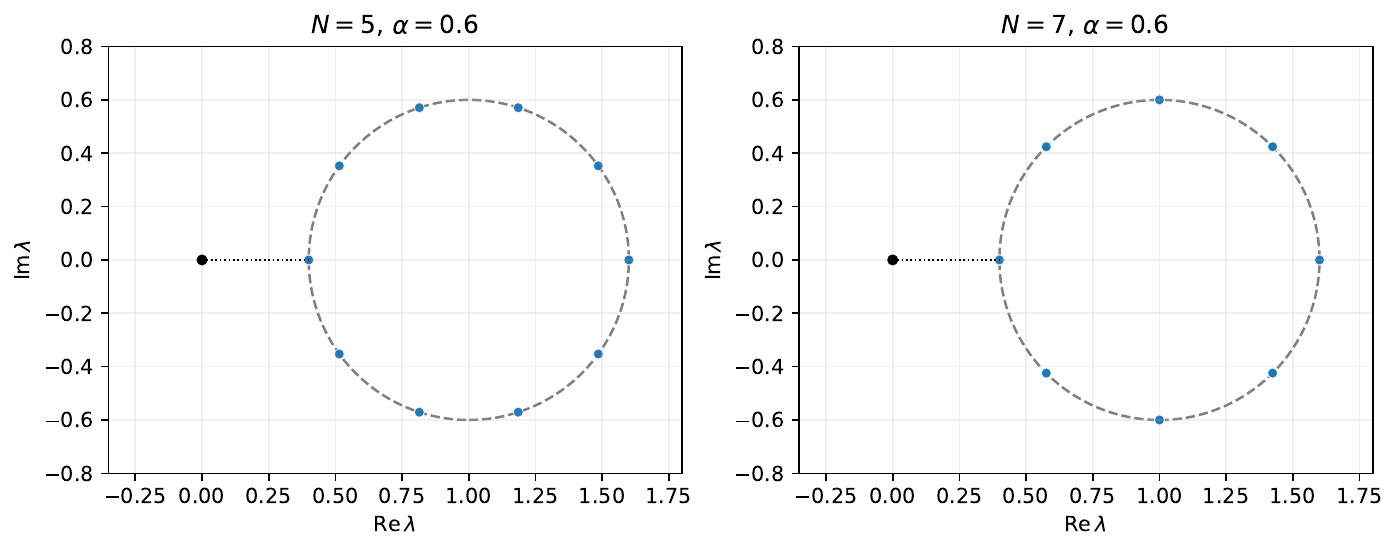}
		\caption{Spectrum of $I-0.6P_N$ for $N\in\{5,7\}$. All eigenvalues lie on the circle $|\lambda-1|=0.6$, and the spectral gap is $1-\alpha=0.4$.}
		\label{fig:spectrum}
	\end{figure}
	
	\section{Periods modulo prime powers}\label{sec:local}
	
	We now determine how point periods change as the modulus grows from $p$ to $p^a$. The main result applies to any infinite-order matrix in $SL_2(\mathbb Z)$ at an odd prime not dividing its discriminant. For the cat map ($\Delta_A=5$) this means every odd prime except $5$.
	
	\subsection{Local data}
	
	Throughout this section, let $A\in SL_2(\mathbb Z)$ have infinite order, and fix an odd prime $p\nmid\Delta_A$. Let $r_p$ be the order of $A$ modulo $p$, and let $e_p\geq1$ be the largest integer such that $A^{r_p}\equiv I\pmod{p^{e_p}}$. The infinite-order assumption ensures that $e_p$ is finite. Thus
	\begin{equation}\label{eq:depth}
		A^{r_p}\equiv I\pmod{p^{e_p}},\qquad A^{r_p}\not\equiv I\pmod{p^{e_p+1}}.
	\end{equation}
	Equivalently, every entry of $A^{r_p}-I$ is divisible by $p^{e_p}$, and at least one is not divisible by $p^{e_p+1}$. For a nonzero integer $m$, let $v_p(m)$ be the exponent of $p$ in its prime factorisation. For the cat map, $r_p=\pi(p)$.
	
	\begin{lemma}[Lift invertibility]\label{lem:invertible-lift}
		The matrix
		\[
		C_p=\frac{A^{r_p}-I}{p^{e_p}}
		\]
		has an inverse modulo $p$.
	\end{lemma}
	\begin{proof}
		Write $A^{r_p}=I+p^{e_p}C_p$, so $C_p\not\equiv0\pmod p$ by the definition of $e_p$. Since $\det A=1$,
		\[
		1=\det(I+p^{e_p}C_p)=1+p^{e_p}\tr(C_p)+p^{2e_p}\det(C_p).
		\]
		Subtracting $1$ and dividing by $p^{e_p}$ gives $\tr(C_p)\equiv0\pmod p$. By Lemma~\ref{lem:diag-mod-p}, $A$ has distinct eigenvalues and is diagonalizable over $\F_{p^2}$. The matrix $C_p$ commutes with $A$ (it is a polynomial in $A$ divided by $p^{e_p}$), so it is diagonal in the same basis: $C_p\equiv\operatorname{diag}(c,c')$. The trace condition gives $c'=-c$, and $C_p\not\equiv0$ forces $c\neq0$. Hence $\det C_p\equiv-c^2\neq0$, and $C_p$ is invertible.
	\end{proof}
	
	\begin{lemma}[Powers near the identity]\label{lem:near-identity}
		Let $k\geq1$ and let $W=I+p^kC$ with $C$ an integer matrix. Then:
		\begin{enumerate}
			\item[(i)] for every $q\geq1$, $W^q=I+p^kC_q$ with $C_q\equiv qC\pmod p$;
			\item[(ii)] $W^p=I+p^{k+1}C'$ with $C'\equiv C\pmod p$.
		\end{enumerate}
	\end{lemma}
	\begin{proof}
		(i) The binomial expansion gives $W^q=I+qp^kC+\sum_{j\geq2}\binom qjp^{jk}C^j$, and every term with $j\geq2$ is divisible by $p^{2k}$, hence by $p^{k+1}$.
		
		(ii) Expanding,
		\[
		W^p-I=p^{k+1}C+\sum_{j=2}^{p-1}\binom pjp^{jk}C^j+p^{pk}C^p.
		\]
		For $2\leq j\leq p-1$, the binomial coefficient is divisible by $p$, so the term is divisible by $p^{1+2k}$, and $1+2k\geq k+2$. The last term is divisible by $p^{pk}$, and $pk\geq3k\geq k+2$ because $p\geq3$ (here we use that $p$ is odd) and $k\geq1$. So every term except the first is divisible by $p^{k+2}$.
	\end{proof}
	
	\begin{lemma}[Valuation formula]\label{lem:lift-exponent}
		For every positive integer $m$,
		\[
		A^{r_pm}-I=p^{e_p+v_p(m)}U_m,
		\]
		where $U_m$ is an integer matrix that is invertible modulo $p$.
	\end{lemma}
	\begin{proof}
		Write $m=p^sq$ with $p\nmid q$, and apply Lemma~\ref{lem:near-identity}(ii) $s$ times to $W=A^{r_p}=I+p^{e_p}C_p$. This gives $A^{r_pp^s}=I+p^{e_p+s}C^{(s)}$ with $C^{(s)}\equiv C_p\pmod p$. Now apply part (i) with exponent $q$: $A^{r_pm}=I+p^{e_p+s}U_m$ with $U_m\equiv qC^{(s)}\equiv qC_p\pmod p$. This is invertible because $p\nmid q$ and $C_p$ is invertible by Lemma~\ref{lem:invertible-lift}.
	\end{proof}
	
	\subsection{Period lifting}
	
	For a nonzero point $x=(x_1,x_2)\in(\mathbb Z/p^a\mathbb Z)^2$, let $j$ be the largest integer in $\{0,\ldots,a-1\}$ for which both coordinates are divisible by $p^j$, and set $b=a-j$. Then $x=p^j y$ for some $y$ with at least one coordinate not divisible by $p$.
	
	\begin{theorem}[Period lifting]\label{thm:local-period}
		Let $p$ be an odd prime with $p\nmid\Delta_A$, and let $r_p,e_p$ be as above. For a nonzero point $x$ modulo $p^a$, define $b$ as above. Then the period of $x$ under $x\mapsto Ax\pmod{p^a}$ is
		\begin{equation}\label{eq:period-piecewise}
			k_{p^a}(x)=
			\begin{cases}
				r_p,&b\leq e_p,\\
				r_pp^{b-e_p},&b>e_p.
			\end{cases}
		\end{equation}
		The origin has period $1$.
	\end{theorem}
	\begin{proof}
		Write $x=p^j y$. The return condition $A^nx\equiv x\pmod{p^a}$ becomes $A^n y\equiv y\pmod{p^b}$, where at least one coordinate of $y$ is not divisible by $p$.
		
		First, $r_p\mid n$ is necessary. Otherwise $A^n-I$ is invertible modulo $p$ by Lemma~\ref{lem:diag-mod-p}, so reducing the congruence modulo $p$ would force $y\equiv0\pmod p$, a contradiction. So $n=r_pm$ for some $m\geq1$.
		
		By Lemma~\ref{lem:lift-exponent}, $A^{r_pm}-I=p^{e_p+v_p(m)}U_m$ with $U_m$ invertible modulo $p$. Then $U_my$ still has a coordinate not divisible by $p$, so $A^{r_pm}y\equiv y\pmod{p^b}$ exactly when $e_p+v_p(m)\geq b$. If $b\leq e_p$, the smallest such $m$ is $1$. If $b>e_p$, the smallest is $m=p^{b-e_p}$. This proves the formula.
	\end{proof}
	
	\begin{corollary}[Local period counts]\label{cor:period-distribution}
		Let $E_{p^a}(k)$ count points of exact period $k$. If $a\leq e_p$, then
		\[
		E_{p^a}(1)=1,\qquad E_{p^a}(r_p)=p^{2a}-1,
		\]
		and all other counts vanish. If $a>e_p$, then
		\begin{align*}
			E_{p^a}(1)&=1,\\
			E_{p^a}(r_p)&=p^{2e_p}-1,\\
			E_{p^a}(r_pp^s)&=p^{2(e_p+s)}-p^{2(e_p+s-1)}\quad (1\leq s\leq a-e_p),
		\end{align*}
		and all other counts vanish.
	\end{corollary}
	\begin{proof}
		For each $b=1,\ldots,a$, there are $p^{2b}$ vectors whose two coordinates are divisible by $p^{a-b}$, and $p^{2(b-1)}$ vectors whose two coordinates are divisible by $p^{a-b+1}$. Their difference, $p^{2b}-p^{2(b-1)}$, is the number of points with this value of $b$. By Theorem~\ref{thm:local-period} these points have period $r_p$ when $b\leq e_p$ and $r_pp^{b-e_p}$ when $b>e_p$. Adding the counts gives the stated formulas.
	\end{proof}
	
	\begin{example}[The cat map modulo $49$]\label{ex:49}
		Take $p=7$. Then $\pi(7)=8$, and by Theorem~\ref{thm:fib-power}
		\[
		A^8-I=\begin{pmatrix}609&987\\987&1596\end{pmatrix}=7\begin{pmatrix}87&141\\141&228\end{pmatrix}.
		\]
		The entry $87$ is not divisible by $7$, so $e_7=1$. Hence, modulo $49$:
		\begin{itemize}
			\item the $48$ nonzero points with both coordinates divisible by $7$ ($b=1\leq e_7$) have period $8$;
			\item the other $49^2-49=2352$ points ($b=2>e_7$) have period $8\cdot7=56$.
		\end{itemize}
		Together with the origin this accounts for $1+48+2352=49^2$ points, that is, one fixed point, $6$ cycles of length $8$ and $42$ cycles of length $56$. The same computation at $p=3$ gives $\pi(3)=4$, $e_3=1$, and modulo $9$ one fixed point, two cycles of length $4$ and six of length $12$.
	\end{example}
	
	\begin{remark}[The value of $e_p$ for the cat map]
		For the cat map, $e_p\geq2$ happens exactly when $\pi_{\mathrm{Fib}}(p^2)=\pi_{\mathrm{Fib}}(p)$ for the Pisano period, that is, when $p$ is a Wall--Sun--Sun prime~\cite{Wall1960}. No such prime is known, so in practice $e_p=1$ and modulo $p^a$ the periods are $\pi(p),\,p\,\pi(p),\ldots,p^{a-1}\pi(p)$.
	\end{remark}
	
	\section{The exceptional prime 5}
	\label{sec:p5}
	
	For the cat-map matrix, the characteristic polynomial is $t^2-3t+1$, with discriminant $5$. The lifting theorem therefore excludes $p=5$. Modulo $5$ the characteristic polynomial is $(t+1)^2$, so $A$ has the single eigenvalue $-1$ and is not diagonalizable; this is why the one-period behaviour of Corollary~\ref{cor:one-period} fails here. We compute the fixed-point and period counts separately.
	
	Let
	\[
	A=\begin{pmatrix}1&1\\1&2\end{pmatrix},\qquad
	B=A^2=\begin{pmatrix}2&3\\3&5\end{pmatrix},\qquad
	X=B-I=\begin{pmatrix}1&3\\3&4\end{pmatrix}.
	\]
	Multiplying these matrices gives
	\begin{equation}\label{eq:p5-identities}
		X^2=5B=5X+5I,\qquad \tr X=5,\qquad \det X=-5.
	\end{equation}
	Modulo $5$, $X$ is nonzero and $X^2=0$, so $B^m=(I+X)^m=I+mX$ for every $m\geq0$. Hence $B$ has exact order $5$ modulo $5$, and $A$ has order $\pi(5)=10$.
	
	\begin{lemma}[Smith factors]\label{lem:p5-SNF}
		Write $m=5^s q$, where $s\geq0$ and $5\nmid q$, and let $d_1\mid d_2$ be the Smith invariant factors of $B^m-I$. Then
		\[
		v_5(d_1)=s,\qquad v_5(d_2)=s+1.
		\]
	\end{lemma}
	\begin{proof}
		\emph{Step 1: powers of $5$.} Put $M_s=B^{5^s}-I$. We claim
		\begin{equation}\label{eq:Ms}
			M_s=5^sY_s,\qquad Y_s\equiv X\pmod5,\qquad v_5(\det M_s)=1+2s.
		\end{equation}
		For $s=0$ we have $M_0=X$ and $\det X=-5$. Assume \eqref{eq:Ms} for $s$. Since $B^{5^{s+1}}=(I+M_s)^5$,
		\[
		M_{s+1}=5M_s+10M_s^2+10M_s^3+5M_s^4+M_s^5=M_sH_s,
		\]
		where $H_s=5I+10M_s+10M_s^2+5M_s^3+M_s^4$. After dividing $M_{s+1}$ by $5^{s+1}$, the first term is $Y_s$, and each of the other four terms is divisible by $5$: for $s\geq1$ this is clear from $M_s=5^sY_s$, and for $s=0$ it follows from $X^2=5B$ (so that $X^2$, $X^3$, $X^4$ and $X^5/5$ are all divisible by $5$). Hence $Y_{s+1}\equiv Y_s\equiv X\pmod5$.
		
		For the determinant, we compute $v_5(\det H_s)=2$. If $s\geq1$, then $H_s/5=I+2M_s+2M_s^2+M_s^3+M_s^4/5\equiv I\pmod5$. If $s=0$, then \eqref{eq:p5-identities} gives $H_0/5\equiv I+2X\pmod5$, and $\det(I+2X)=1+2\tr X+4\det X=1+10-20\equiv1\pmod5$. In both cases $v_5(\det H_s)=2$, so $v_5(\det M_{s+1})=v_5(\det M_s)+2$.
		
		Since $Y_s\not\equiv0\pmod5$, the greatest common divisor of the entries of $M_s$ has $5$-adic valuation exactly $s$; that is, $v_5(d_1)=s$ for $m=5^s$. The two valuations sum to $v_5(\det M_s)=1+2s$, so $v_5(d_2)=s+1$.
		
		\emph{Step 2: general $m$.} For $m=5^sq$ with $5\nmid q$,
		\[
		B^m-I=M_s\bigl(I+B^{5^s}+\cdots+B^{(q-1)5^s}\bigr).
		\]
		The second factor is invertible modulo $5$. If $s\geq1$, then $B^{5^s}\equiv I$ and the factor is congruent to $qI$. If $s=0$, it is congruent to $qI+\binom q2X$, whose determinant is $q^2+q\binom q2\tr X+\binom q2^2\det X\equiv q^2\not\equiv0\pmod5$, using $\tr X=5$ and $\det X=-5$. Multiplying by a matrix with determinant prime to $5$ does not change the $5$-adic valuations of the Smith invariant factors. So the exponents are $s$ and $s+1$ for every $m$.
	\end{proof}
	
	\begin{theorem}[Fixed-point counts modulo $5^a$]\label{thm:p5-fixed}
		Let $a\geq1$, $n\geq1$, and $F_{5^a}(n)=\#\{x\in(\Z/5^a\Z)^2:A^nx\equiv x\pmod{5^a}\}$. For odd $n$, $F_{5^a}(n)=1$. For even $n=2m$, write $s=v_5(m)$. Then
		\begin{equation}\label{eq:p5-fixed}
			F_{5^a}(2m)=
			\begin{cases}
				5^{2s+1},&s<a,\\
				5^{2a},&s\geq a.
			\end{cases}
		\end{equation}
	\end{theorem}
	\begin{proof}
		Suppose $n$ is odd. Over the algebraic closure of $\F_5$, $A$ has the single eigenvalue $-1$, so $A^n-I$ has the single eigenvalue $(-1)^n-1=-2$. Thus $\det(A^n-I)\equiv4\pmod5$, so $A^n-I$ is invertible modulo $5$, hence modulo $5^a$, and the origin is the only fixed point.
		
		Now let $n=2m$, so $A^n-I=B^m-I$. By Lemma~\ref{lem:p5-SNF} and Lemma~\ref{lem:snf-kernel},
		\[
		F_{5^a}(2m)=5^{\min(s,a)+\min(s+1,a)},
		\]
		which is the piecewise expression \eqref{eq:p5-fixed}.
	\end{proof}
	
	\begin{corollary}[Exact period counts]\label{cor:p5-periods}
		For the cat map modulo $5^a$, the only exact periods are $1$ and $2\cdot5^j$ for $0\leq j\leq a$. Their multiplicities are
		\begin{align*}
			E_{5^a}(1)&=1,\\
			E_{5^a}(2)&=4,\\
			E_{5^a}(2\cdot5^j)&=120\cdot25^{j-1},\qquad 1\leq j<a,\\
			E_{5^a}(2\cdot5^a)&=20\cdot25^{a-1}.
		\end{align*}
		Consequently,
		\begin{align*}
			\cT_{5^a}(\alpha)
			={}&\frac{1}{1-\alpha}+\frac{4}{1-\alpha^2}
			+\sum_{j=1}^{a-1}\frac{120\cdot25^{j-1}}{1-\alpha^{2\cdot5^j}}\\
			&+\frac{20\cdot25^{a-1}}{1-\alpha^{2\cdot5^a}}.
		\end{align*}
	\end{corollary}
	\begin{proof}
		Write $F_j=F_{5^a}(2\cdot5^j)$, so $F_j=5^{2j+1}$ for $j<a$ and $F_j=5^{2a}$ for $j\geq a$. Let $x\neq0$ have period $k$. Then $x$ is a nonzero fixed point of $A^k$, so $k$ is even by Theorem~\ref{thm:p5-fixed}. Moreover $F_a=5^{2a}$ means $A^{2\cdot5^a}\equiv I$, so $k\mid2\cdot5^a$. Hence $k=2\cdot5^j$ with $0\leq j\leq a$.
		
		The fixed points of $A^{2\cdot5^j}$ are exactly the points of period $1$ or $2\cdot5^i$ with $i\leq j$. Therefore $E_{5^a}(2\cdot5^j)=F_j-F_{j-1}$ for $j\geq1$, and $E_{5^a}(2)=F_0-1=4$. For $1\leq j<a$ this is $5^{2j+1}-5^{2j-1}=120\cdot25^{j-1}$, and for $j=a$ it is $5^{2a}-5^{2a-1}=20\cdot25^{a-1}$. The trace formula follows from \eqref{eq:trace-period-sum}.
	\end{proof}
	
	\begin{example}[Moduli $5$ and $25$]\label{ex:5-25}
		For $a=1$: $E_5(1)=1$, $E_5(2)=4$, $E_5(10)=20$, so there are $1+2+2=5$ cycles and
		\[
		\cT_5(\alpha)=\frac1{1-\alpha}+\frac4{1-\alpha^2}+\frac{20}{1-\alpha^{10}}.
		\]
		The fixed points of $A^2$ modulo $5$ form the line $\{0,(2,1),(4,2),(1,3),(3,4)\}$, and $A$ swaps $(2,1)\leftrightarrow(3,4)$ and $(4,2)\leftrightarrow(1,3)$: these are the two $2$-cycles. The remaining $20$ points form two $10$-cycles.
		
		For $a=2$: $E_{25}(1)=1$, $E_{25}(2)=4$, $E_{25}(10)=120$, $E_{25}(50)=500$, and $1+4+120+500=625=25^2$.
	\end{example}
	
	\begin{corollary}[Fibonacci valuation]\label{cor:fib5}
		For every positive integer $m$,
		\[
		v_5(F_{2m})=v_5(m).
		\]
	\end{corollary}
	\begin{proof}
		By Theorem~\ref{thm:fib-power}, $\tr(A^{2m})=F_{4m-1}+F_{4m+1}$, and the standard identity $F_{4m-1}+F_{4m+1}=2+5F_{2m}^2$ gives
		\[
		\det(A^{2m}-I)=2-\tr(A^{2m})=-5F_{2m}^2.
		\]
		Lemma~\ref{lem:p5-SNF} gives $v_5(\det(A^{2m}-I))=v_5(d_1d_2)=1+2v_5(m)$. Comparing valuations yields $1+2v_5(F_{2m})=1+2v_5(m)$. For classical background on Fibonacci congruences, see Wall~\cite{Wall1960}.
	\end{proof}
	
	\section{Combining prime-power moduli}
	\label{sec:crt}
	
	The Chinese remainder theorem identifies the cat map modulo $N$ with its component maps modulo the prime powers dividing $N$. Fixed-point counts multiply, while point periods combine by least common multiples.
	
	\begin{proposition}[Order modulo $N$]\label{prop:order-crt}
		If $N=\prod_{i=1}^t p_i^{a_i}$ is the prime factorisation of $N$, then
		\[
		\pi(N)=\lcm_{1\leq i\leq t}\pi(p_i^{a_i}).
		\]
	\end{proposition}
	\begin{proof}
		The congruence $A^m\equiv I\pmod N$ holds exactly when it holds modulo every prime power dividing $N$. The smallest exponent that works for all of them is their least common multiple.
	\end{proof}
	
	\begin{theorem}[CRT counting formulas]\label{thm:crt}
		For $N=\prod_{i=1}^t p_i^{a_i}$ and every $n\geq1$,
		\begin{equation}\label{eq:crt-fixed}
			F_N(n)=\prod_{i=1}^tF_{p_i^{a_i}}(n).
		\end{equation}
		If $E_{p_i^{a_i}}(\ell_i)$ denotes the number of points of exact period $\ell_i$ modulo $p_i^{a_i}$, then
		\begin{equation}\label{eq:crt-period}
			E_N(k)=\sum_{\substack{\ell_1,\ldots,\ell_t\geq1\\\lcm(\ell_1,\ldots,\ell_t)=k}}
			\prod_{i=1}^tE_{p_i^{a_i}}(\ell_i).
		\end{equation}
		The trace for modulus $N$ is therefore
		\[
		\cT_N(\alpha)=\sum_{k\geq1}\frac{E_N(k)}{1-\alpha^k},
		\]
		with $E_N(k)$ obtained from the local distributions by \eqref{eq:crt-period}.
	\end{theorem}
	\begin{proof}
		The Chinese remainder theorem gives a one-to-one correspondence
		\[
		(\Z/N\Z)^2\cong\prod_{i=1}^t(\Z/p_i^{a_i}\Z)^2
		\]
		under which the map acts separately on each prime-power component. A point returns after $n$ steps exactly when every component does, giving \eqref{eq:crt-fixed}. A tuple of points returns for the first time after the least common multiple of its component periods, which gives \eqref{eq:crt-period}. The trace formula is \eqref{eq:trace-period-sum}.
	\end{proof}
	
	\begin{example}[The cat map modulo $15$]\label{ex:15}
		Modulo $3$ the exact-period counts are $E_3(1)=1$, $E_3(4)=8$ (Example~\ref{ex:N3}), and modulo $5$ they are $E_5(1)=1$, $E_5(2)=4$, $E_5(10)=20$ (Example~\ref{ex:5-25}). Combining by \eqref{eq:crt-period}:
		\begin{center}
			\begin{tabular}{@{}cccr@{}}
				\toprule
				period mod $3$ & period mod $5$ & period mod $15$ & number of points\\
				\midrule
				$1$ & $1$ & $1$ & $1\cdot1=1$\\
				$1$ & $2$ & $2$ & $1\cdot4=4$\\
				$1$ & $10$ & $10$ & $1\cdot20=20$\\
				$4$ & $1$ & $4$ & $8\cdot1=8$\\
				$4$ & $2$ & $4$ & $8\cdot4=32$\\
				$4$ & $10$ & $20$ & $8\cdot20=160$\\
				\bottomrule
			\end{tabular}
		\end{center}
		So $E_{15}(1)=1$, $E_{15}(2)=4$, $E_{15}(4)=8+32=40$, $E_{15}(10)=20$, $E_{15}(20)=160$, with total $225=15^2$, and
		\[
		\cT_{15}(\alpha)=\frac1{1-\alpha}+\frac4{1-\alpha^2}+\frac{40}{1-\alpha^4}+\frac{20}{1-\alpha^{10}}+\frac{160}{1-\alpha^{20}}.
		\]
		The number of cycles is $C_{15}=1+2+10+2+8=23$. As a consistency check, $F_{15}(2)=F_3(2)F_5(2)=1\cdot5=5=E_{15}(1)+E_{15}(2)$, and $F_{15}(4)=9\cdot5=45=1+4+40$.
	\end{example}
	
	\begin{remark}
		For a composite modulus, the fixed-point counts multiply across prime powers at each time, while exact periods combine through least common multiples. The landscape traces themselves do not multiply. For example, $\cT_3(\alpha)\cT_5(\alpha)$ has a double pole at $\alpha=1$ (with leading coefficient $3\cdot5=15$), whereas $\cT_{15}(\alpha)$ has only a simple pole, with residue $C_{15}=23$.
	\end{remark}
	
	\begin{remark}
		Sections~\ref{sec:local} and~\ref{sec:p5} give the period distribution modulo $p^a$ for every odd prime $p$ (for $p\neq5$, once $e_p$ is known). Combined with Theorem~\ref{thm:crt}, this determines the full period distribution, and hence the landscape trace, for every odd modulus $N$.
	\end{remark}
	
	\section{Further questions}
	
	The results give three descriptions of the same finite dynamics: pointwise landscape values, exact-period counts, and the trace and determinant of the associated operator. The local lifting theorem gives period data at primes outside the characteristic discriminant, the calculation at $5$ treats the exceptional prime for the cat map, and the Chinese remainder theorem combines the local distributions.
	
	Several questions remain open. Can analogous lifting formulas be proved for other integer matrices, especially at primes dividing the characteristic discriminant? Can one determine the exact-period distribution modulo $2^a$? Is $e_p=1$ for every prime $p\neq2,5$, that is, are there no Wall--Sun--Sun primes? Which arithmetic properties of the matrix can be recovered from the landscape distribution or from the trace as a function of $\alpha$?
	
	\section*{Acknowledgements}
	The author thanks Dr. Sudhir Ranjan Jain and Dr. Nishchal Dwivedi for guidance.

\end{document}